\documentclass[preprint,12pt]{elsarticle}
\usepackage{amsfonts}
\usepackage{amsmath}
\usepackage{amssymb}
\usepackage{graphicx}
\usepackage[tableposition=top]{caption}
\usepackage{subcaption}
\usepackage{float}
\usepackage{setspace}
\usepackage{placeins}

\graphicspath{{C:/Figures/}}
\newtheorem{theorem}{Theorem}
\newtheorem{proof}{Proof}

\newtheorem{definition}{Definition}

\newtheorem{lemma}{Lemma}

\numberwithin{equation}{section}
\journal{...}

\begin{document}

\begin{frontmatter}
\title{On the complete synchronization of a time--fractional reaction--diffusion system with the Newton--Leipnik nonlinearity}
\author[a]{Djamel Mansouri}
\author[b]{Samir Bendoukha}
\author[c]{Salem Abdelmalek}
\author[d]{Amar Youkana}
\address[a]{Department of Mathematics, ICOSI laboratory, University Abbes Laghrour, Khenchela, Algeria. Email: mansouridjamel11@yahoo.fr}
\address[b]{Electrical Engineering Department, College of Engineering at Yanbu, Taibah University, Saudi Arabia. E-mail: sbendoukha@taibahu.edu.sa}
\address[c]{Department of Mathematics and Computer Science, University of Larbi Tebessi,Tebessa, 12002 Algeria. Email: sallllm@gmail.com}
\address[d]{Department of Mathematics, University of Batna $2$, Batna, Algeria. Email: youkana\_amar@yahoo.fr}

\begin{abstract}
In this paper, we consider a time--fractional reaction--diffusion system with the same nonlinearities of the Newton--Leipnik chaotic system. Through analytical tools and numerical results, we derive sufficient conditions for the asymptotic stability of the proposed model and show the existence of chaos. We also propose a nonlinear synchronization controller for a pair of systems and establish the local and global asymptotic convergence of the trajectories by means of fractional stability theory and the Lyapunov method.
\end{abstract}
\end{frontmatter}

\section{Introduction\label{SecIntro}}

Chaos has become a very common term in a number of scientific and
engineering disciplines. Over the last three decades the amount of research
publications dealing with chaotic dynamical systems and their applications
is in the thousands. One of the main reasons behind this interest is the
chaotic nature of several physical phenomena such as the weather or the
turbulent flow of fluids. Another important reason is that in certain fields
such as secure communications and data encryption, the random nature of the
chaotic system's states is a desirable property \cite%
{Kocarev2001,Dachselt2001,Masuda2002,Lawande2005,Masuda2006}. Chaotic
systems are generally characterized by their high sensitivity to variations
in the initial conditions, which can be attributed to the system having at
least one positive Lyapunov exponent. If two identical systems with the
exact same parameters start from very close points in phase space, they end
up following trajectories that move away from one another at an exponential
rate. As a result, the trajectories of such systems seem random but are
completely deterministic and can be easily replicated if the exact initial
conditions are known. Perhaps, the usefulness of chaos became more apparent
once their synchronization became possible. Synchronization refers to the
control of one chaotic system to follow the exact same trajectory of another
through adaptive rules. Among the first studies that realized the
possibility of such a controller are \cite%
{Yamada1983,Yamada1984,Afraimovich1983,Pecora1990}.

Originally, chaotic dynamical systems were considered with integer
differentiation orders. However, researchers quickly realized that
fractional calculus can improve the modeling of natural phenomena and lead
to a wider range of system dynamics and attractor types. Fractional calculus
is an old topic as it goes back to the seventeenth century. However, its
importance in modeling was only recently observed. Fractional--order systems
come with some added implementation complexity but at the same time offer a
higher level of flexibility and a wider range of chaotic trajectories. This
is mainly due to the fact that unlike integer--order differentiation, its
fractional counterpart comes with an infinite memory and thus takes into
consideration all previous states. The synchronization of fractional chaotic
systems is widely considered, see for instance \cite%
{Martinez-Guerra2014,Mahmoud2016,Maheri2016}.

Another piece of background that is important to us in this paper is that of
chaotic reaction--diffusion systems. It is easy to see that the vast
majority of studies dealing with chaos consider an ODE system where the
dependent variables represent the time evolution of certain physical
quantities. However, in \cite{Cross1993}, the authors pointed out that chaos
is particularly important in modeling and understanding the laminar and
turbulent flow of fluids. They argued that since fluid consists of a
continuum of hydrodynamic modes, it is more suitable to describe its
dynamics by means of a spatially extended system of differential equations,
i.e. a reaction--diffusion system. They studied the dynamics of the complex
Ginzburg--Landau and Kuramoto--Sivashinsky equations. Recently after that,
the authors of \cite{Lai1994} showed that the general chaotic behavior of a
reaction--diffusion system is similar to the ODE case in the sense that the
system is extremely sensitive to changes in the intial states as well as the
system's bifurcation parameters. Parekh \textit{et al.} \cite{Parekh1997}
studied the control of an autocatalytic reaction--diffusion system. They
devised a synchronization scheme and established the convergence of the
error by means of appropriate Lyapunov functionals. An interesting summary
of chaos in reaction--diffusion systems is given in \cite{Zelik2007}.
Several studies have been published recently dealing with the stabilization
and synchronization of spatio--temporal chaotic systems. For instance, it
has been shown that neural networks can exhibit chaotic dynamics \cite%
{Wang2007,Yu2011,Yang2013}. Other types of systems that may under certain
circumstances be chaotic include predator--prey models \cite{Hu2015} and the
FitzHugh--Nagumo model \cite{Zaitseva2016,Zaitseva2017}.

This paper is concerned with the chaotic dynamics and synchronization of a
reaction--diffusion system that assumes the same nonlinearities of the
Newton--Leipnik system first proposed in \cite{Leipnik1981} as a model of
the rigid body motion through linear feedback (LFRBM). The dynamics of the
original system as well as its control were studied in \cite%
{Wolf1985,Qiang2008,Jovic2011,Bendoukha2018b}. In \cite{Bendoukha2018a}, the
authors examine a reaction--diffusion version of the system. Kang \textit{et
al.} proposed a fractional version of the system in \cite{Kang2007}. They
investigated the fractional model numerically over wide parameter ranges and
commented on the impact of parameters on the system dynamics. The same
system was examined further in \cite{Sheu2008}\ and shown to exhibit complex
dynamics including fixed points, periodic motions, chaotic motions, and
transient chaos. In \cite{Khan2017}, the authors develop a disturbance
observer based adaptive sliding mode hybrid projective synchronization
scheme for the fractional system. In our work, we consider a combination of
the above mentioned properties in a single system. The proposed
fractional--time reaction--diffusion Newton--Leipnik system is investigated
analytically and experimentally in terms of its dynamics and synchronization.

\section{On Fractional Calculus and Stability}

Since the reader may not be very familiar with some of the notation and
terminology used throughout this paper with regards to fractional calculus,
it seems suitable to provide the following definitions and lemma along with
helpful references.

\begin{definition}
\label{Def1}\cite{Kilbas2006} The Liouville fractional derivative of order $%
\delta$\ of an integrable function $f\left( t\right) $ is defined as%
\begin{equation}
_{t_{0}}D_{t}^{-\delta}f\left( t\right) =\frac{1}{\Gamma \left( \delta
\right) }\int_{t_{0}}^{t}\frac{f\left( \tau \right) }{\left( t-\tau \right)
^{1-\delta}}d\tau.  \label{1.1}
\end{equation}
\end{definition}

\begin{definition}
\label{Def2}\cite{Kilbas2006} The Caputo fractional derivative of order $%
\delta>0$ of a function $f\ $of class $C^{n}$ for $t>t_{0}$ is defined as%
\begin{equation}
_{t_{0}}^{C}D_{t}^{\delta}f\left( t\right) =\frac{1}{\Gamma \left( n-\delta
\right) }\int_{t_{0}}^{t}\frac{f^{\left( n\right) }\left( \tau \right) }{%
\left( t-\tau \right) ^{\delta-n-1}}d\tau,  \label{1.2}
\end{equation}
with $n=\min \left \{ k\in%
\mathbb{N}
\ |\ k>\delta \right \} $ and $\Gamma$\ representing the gamma function.
\end{definition}

\begin{definition}
\label{Def3}\cite{Li2010} The constant $x_{0}$ is considered to be an
equilibrium for the Caputo fractional nonautonomous dynamic system%
\begin{equation}
_{t_{0}}^{C}D_{t}^{\delta }x\left( t\right) =f\left( t,x\right) ,
\label{1.3}
\end{equation}%
if and only if 
\begin{equation}
f\left( t,x_{0}\right) =0.  \label{1.4}
\end{equation}
\end{definition}

\begin{lemma}
\label{Lemma1}\cite{AguilaCamacho2014} Let $x\left( t\right) $ be a
continuous and differentiable real function. For any time instant $t\geq
t_{0}$,%
\begin{equation}
_{t_{0}}^{C}D_{t}^{\delta }x^{2}\left( t\right) \leq 2x\left( t\right)
_{t_{0}}^{C}D_{t}^{\delta }x\left( t\right) ,  \label{1.5}
\end{equation}%
with $\delta \in \left( 0,1\right) $.
\end{lemma}

We say that the constant $\left( u^{\ast },v^{\ast }\right) $ is an
equilibrium for the Caputo fractional non--autonomous dynamic system%
\begin{equation}
\left \{ 
\begin{array}{l}
_{t_{0}}^{C}D_{t}^{\delta _{1}}u=F\left( u,v\right) ,\text{ \  \  \  \  \  \ in }%
\mathbb{R}^{+},\medskip \\ 
_{t_{0}}^{C}D_{t}^{\delta _{2}}v=G\left( u,v\right) ,\text{ \  \  \  \  \ in }%
\mathbb{R}^{+},%
\end{array}%
\right.  \label{1.5.1}
\end{equation}%
if and only if 
\begin{equation}
F\left( u^{\ast },v^{\ast }\right) =G\left( u^{\ast },v^{\ast }\right) =0.
\label{1.5.2}
\end{equation}%
Assuming a fractional order system comprising of two differential equations
with an equilibrium $\left( u^{\ast },v^{\ast }\right) $ and Jacobian matrix 
$J\left( u^{\ast },v^{\ast }\right) $ evaluated at $\left( u^{\ast },v^{\ast
}\right) $, the following local and global asymptotic stability results are
important.

\begin{lemma}
\label{Lemma2}\cite{Matignon1996} Assuming $\delta _{1}=\delta _{2}=\delta $%
, the equilibrium point $\left( u^{\ast },v^{\ast }\right) $ is locally
asymptotically stable iff%
\begin{equation}
\left \vert \arg \left( \lambda _{i}\right) \right \vert >\frac{\delta \pi }{%
2},\  \  \ i=1,2,  \label{1.6}
\end{equation}%
where $\lambda _{i}$ are the eigenvalues of $J\left( u^{\ast },v^{\ast
}\right) $ and $\arg \left( \cdot \right) $ denotes the argument of a
complex number.
\end{lemma}

\begin{lemma}
\label{Lemma3}\cite{Deng2007} Assuming $\delta _{i}=\frac{l_{i}}{m_{i}}%
,i=1,2,$\ with $\left( l_{i},m_{i}\right) =1$ and $l_{i},m_{i}\in 
\mathbb{N}
$, the equilibrium point $\left( u^{\ast },v^{\ast }\right) $ is locally
asymptotically stable iff all the roots $\lambda $ of the characteristic
equation%
\begin{equation}
\det \left( \text{diag}\left( \lambda ^{m\delta _{1}},\lambda ^{m\delta
_{2}}\right) -J\left( u^{\ast },v^{\ast }\right) \right) =0,  \label{1.7}
\end{equation}%
satisfy%
\begin{equation}
\left \vert \arg \left( \lambda \right) \right \vert >\frac{\pi }{2m},
\label{1.8}
\end{equation}%
where $m$ is the least common multiple of the denominators $m_{i}$.
\end{lemma}

\begin{lemma}
\label{Lemma4}If there exists a positive definite Lyapunov function 
\begin{equation}
V(U)=\frac{1}{2}U^{T}\left( t\right) U\left( t\right)  \label{1.9}
\end{equation}%
such that 
\begin{equation}
D_{t}^{\overline{\delta }}V\left( U\right) <0  \label{1.10}
\end{equation}%
for all $t\geq t_{0}$, then the trivial solution of system 
\begin{equation}
D_{t}^{\overline{\delta }}U=F\left( U\right) ,  \label{1.11}
\end{equation}%
where $F:%
\mathbb{R}
^{n}\rightarrow 
\mathbb{R}
^{n}$, $U=\left( u_{1},u_{2},...,u_{n}\right) $, $\overline{\delta }=\left(
\delta _{1},\delta _{2},...,\delta _{n}\right) $, and\ $0<\delta _{i}\leq 1$%
,\ is globally asymptotically stable.
\end{lemma}

\begin{lemma}
\label{Lemma5}If $U\left( t\right) \in 
\mathbb{R}
^{n}$, $\overline{\delta }=\left( \delta _{1},\delta _{2},...,\delta
_{n}\right) $, and $0<\delta _{i}\leq 1$, then%
\begin{equation}
\frac{1}{2}D_{t}^{\overline{\delta }}U^{T}\left( t\right) U\left( t\right)
\leq U^{T}\left( t\right) D_{t}^{\overline{\delta }}U\left( t\right) .
\label{1.12}
\end{equation}
\end{lemma}

\begin{lemma}
\label{Lemma6}Consider the fractional--order system%
\begin{equation}
D^{\beta }\varphi \left( t\right) =f\left( \varphi \left( t\right) \right) ,
\label{1.13}
\end{equation}%
where $0<\delta <1$, with $\varphi \left( t\right) $ $\in $ $%
\mathbb{R}
$ and $\varphi ^{\ast }=0$ as its equilibrium. If for any $\varphi \left(
t\right) $,%
\begin{equation}
\varphi \left( t\right) f\left( \varphi \left( t\right) \right) \leq 0,
\label{1.14}
\end{equation}%
then $\varphi ^{\ast }$ is asymptotically stable. Moreover, if for any $%
\varphi \left( t\right) \neq 0$,%
\begin{equation}
\varphi \left( t\right) f\left( \varphi \left( t\right) \right) <0,
\label{1.15}
\end{equation}%
then $\varphi ^{\ast }$ is asymptotically stable.
\end{lemma}

\section{Standard and Fractional Newton--Leipnik Models\label{SecModel}}

The Newton--Leipnik system first proposed in \cite{Leipnik1981} is of the
form%
\begin{equation}
\left \{ 
\begin{array}{l}
\frac{du_{1}}{dt}=-au_{1}+u_{2}+10u_{2}u_{3}:=f_{1}\left(
u_{1},u_{2},u_{3}\right) , \\ 
\frac{du_{2}}{dt}=-u_{1}-0.4u_{2}+5u_{1}u_{3}:=f_{2}\left(
u_{1},u_{2},u_{3}\right) , \\ 
\frac{du_{3}}{dt}=\alpha u_{3}-5u_{1}u_{2}:=f_{3}\left(
u_{1},u_{2},u_{3}\right) ,%
\end{array}%
\right.  \label{2.1}
\end{equation}%
$u_{i},i=1,2,3,$ denote the system's the states and $a$ and $\alpha $ are
bifurcation parameters. It is easy to show that system (\ref{2.1}) is
dissipative. By taking the divergence of the vector field $f$ on $%
\mathbb{R}
^{3}$, we to obtain (see \cite{Qiang2008})%
\begin{equation}
\text{div}f=\alpha -a-0.4.  \label{2.2}
\end{equation}%
Let $\Omega $ be some region within $%
\mathbb{R}
^{3}$ with a smooth boundary $\partial \Omega $ and let $\Omega \left(
t\right) =\Phi _{t}\left( \Omega \right) $, where $\Phi _{t}$ is the flow of
the vector field $f$, and $V\left( t\right) $ denote the volume of $\Omega
\left( t\right) $. Using Liouville's theorem, we obtain%
\begin{eqnarray}
\frac{dV\left( t\right) }{dt} &=&\int_{\Omega \left( t\right) }\left( \text{%
div}f\right) du_{1}du_{2}du_{3}  \notag \\
&=&\left( \alpha -a-0.4\right) V\left( t\right) .  \label{2.3}
\end{eqnarray}%
Solving the differential equation yields%
\begin{equation}
V\left( t\right) =V\left( 0\right) e^{\left( \alpha -a-0.4\right) t}.
\label{2.4}
\end{equation}%
It is easy to see that assuming $\alpha -a-0.4<0$, the volume decays to zero
asymptotically as $t\rightarrow \infty $, which means that the system is a
dissipative one. It is well known that a dissipative chaotic system has a
strange attractor.

In order to determine the equilibria of system (\ref{2.1}), we set the time
derivatives to zero and solve for $\left( u_{1},u_{2},u_{3}\right) $. This
yields the five equilibria (see \cite{Qiang2008})%
\begin{eqnarray}
O_{1} &=&\left( 
\begin{array}{c}
0 \\ 
0 \\ 
0%
\end{array}%
\right) ,O_{2}=\left( 
\begin{array}{c}
-0.031549 \\ 
0.12238 \\ 
-0.11031%
\end{array}%
\right) ,O_{3}=\left( 
\begin{array}{c}
0.031549 \\ 
-0.12238 \\ 
-0.11031%
\end{array}%
\right) ,  \notag \\
O_{4} &=&\left( 
\begin{array}{c}
0.23897 \\ 
0.030803 \\ 
0.21031%
\end{array}%
\right) ,\text{ \ and \ }O_{5}=\left( 
\begin{array}{c}
-0.23897 \\ 
-0.030803 \\ 
0.21031%
\end{array}%
\right) .  \label{2.5}
\end{eqnarray}

The asymptotic stability of these equilibria can be examined by means of
standard stability results. The system has been shown to exhibit a chaotic
behavior for specific values of the parameters $\left( a,\alpha \right) $.
For instance, subject to parameters $\left( a,\alpha \right) =\left(
0.4,0.175\right) $ and initial conditions $\left( u_{1}\left( 0\right)
,u_{2}\left( 0\right) ,u_{3}\left( 0\right) \right) =\left(
0.349,0,-0.3\right) $, the phase portraits of the system are depicted in
Figure \ref{Figure1}. The system exhibits a strange attractor with two
equilibria.

\begin{figure}[tbph]
\centering \includegraphics[width=\textwidth]{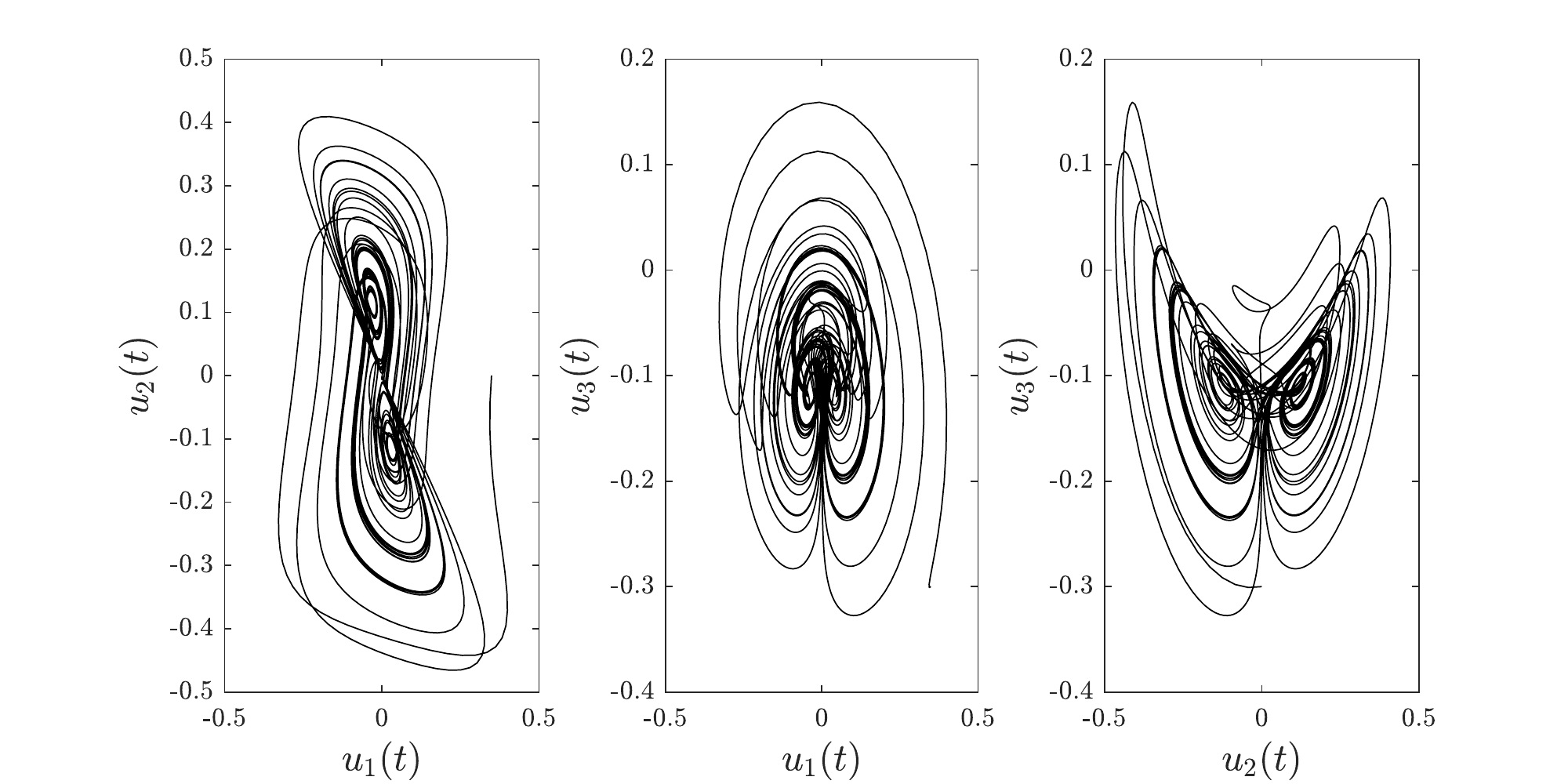}
\caption{Phase--space portraits of the standard Newton--Leipnik chaotic
system for $\left( a,\protect \alpha \right) =\left( 0.4,0.175\right) $ and
initial conditions $\left( u_{1}\left( 0\right) ,u_{2}\left( 0\right)
,u_{3}\left( 0\right) \right) =\left( 0.349,0,-0.3\right) $.}
\label{Figure1}
\end{figure}

The Caputo--type fractional version of the Newton--Leipnik system was
formulated in \cite{Kang2007} as%
\begin{equation}
\left \{ 
\begin{array}{l}
_{t_{0}}^{C}D_{t}^{\delta _{1}}u_{1}=-au_{1}+u_{2}+10u_{2}u_{3}, \\ 
_{t_{0}}^{C}D_{t}^{\delta _{2}}u_{2}=-u_{1}-0.4u_{2}+5u_{1}u_{3}, \\ 
_{t_{0}}^{C}D_{t}^{\delta _{3}}u_{3}=\alpha u_{3}-5u_{1}u_{2},%
\end{array}%
\right.  \label{2.6}
\end{equation}%
where $0<\delta _{1},\delta _{2},\delta _{3}\leq 1$ are the fractional
differentiation orders and $_{t_{0}}^{C}D_{t}^{\delta }$ denotes the Caputo
fractional derivative over $\left( t_{0},\infty \right) $ as defined in (\ref%
{1.2}). Assuming identical orders $\delta _{1}=\delta _{2}=\delta _{3}$, it
has been shown in \cite{Bendoukha2018b} that equilibrium $O_{1}$ is always
unstable whereas $O_{2}$ and $O_{3}$ are asymptotically stable subject to%
\begin{equation}
\delta <0.93660,  \label{2.7}
\end{equation}%
and $O_{4}$ and $O_{5}$ are asymptotically stable subject to%
\begin{equation}
\delta <0.9541.  \label{2.8}
\end{equation}%
These is confirmed by the numerical results in Figure \ref{Figure2}, where
the same parameters and initial conditions assumed previously were adopted
and the fractional order is varied.

As for the incommensurate case where the three orders are non--identical, no
exact bound has been found for the asymptotic stability of the system. It
was, however, shown in \cite{Bendoukha2018b} that for lower orders, the
system is asymptotically stable and for orders close to one, the system
becomes chaotic. For instance, it was shown that for $\left( \delta
_{1},\delta _{2},\delta _{3}\right) =\left( 1,0.95,0.975\right) $, all
euilibria are asymptotically unstable, whereas for $\left( \delta
_{1},\delta _{2},\delta _{3}\right) =\left( 0.85,0.9,0.8\right) $, all of $%
O_{2}$, $O_{3}$, $O_{4}$ and $O_{5}$ are asymptotically stable. This, again,
may be verified through numerical simulations as depicted in Figure \ref%
{Figure3}.

\begin{figure}[tbph]
\centering \includegraphics[width=4.5in]{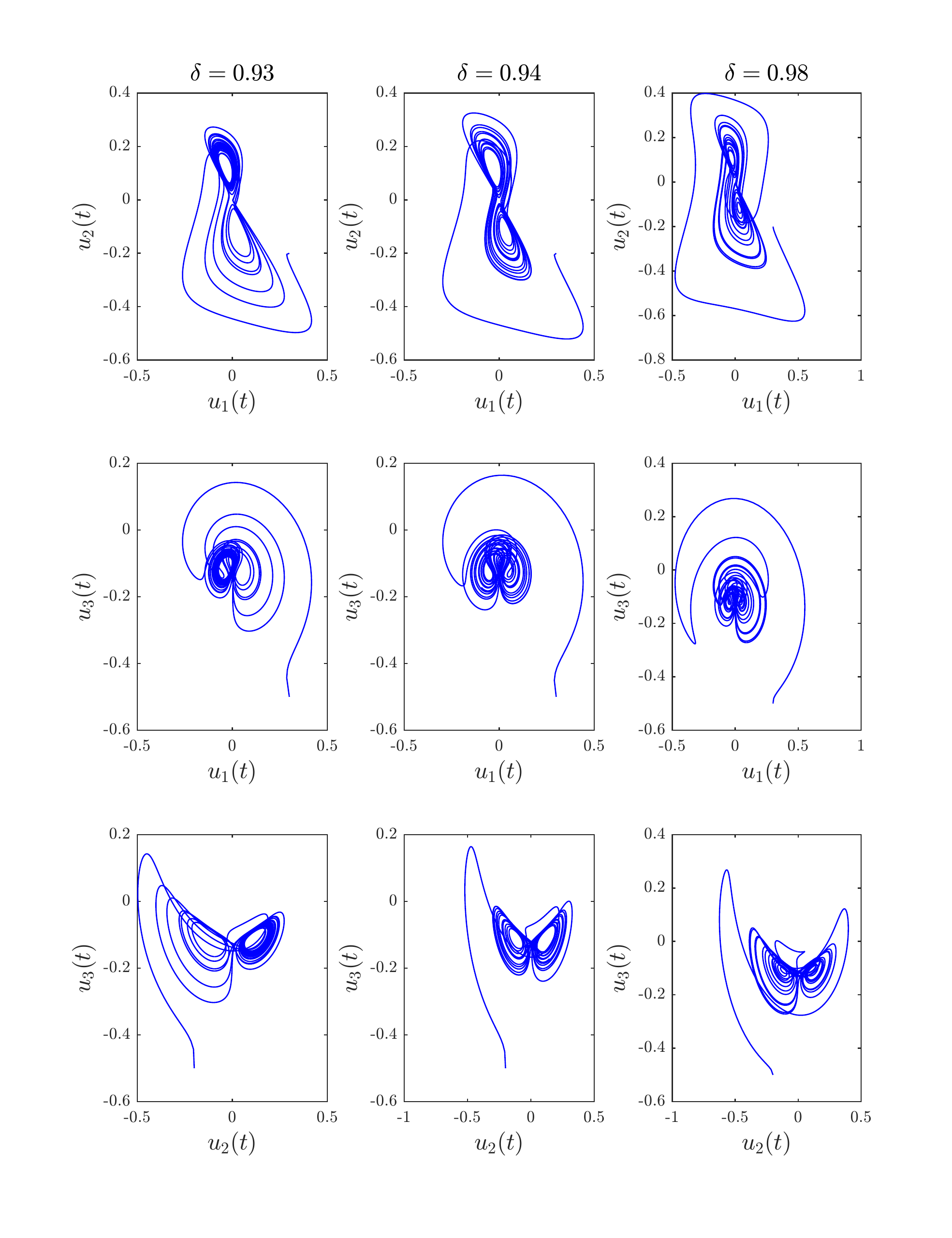}
\caption{Phase--space portraits of the fractional commensurate
Newton--Leipnik chaotic system for $\left( a,\protect \alpha \right) =\left(
0.4,0.175\right) $ and initial conditions $\left( u_{1}\left( 0\right)
,u_{2}\left( 0\right) ,u_{3}\left( 0\right) \right) =\left(
0.349,0,-0.3\right) $ with different fractional orders.}
\label{Figure2}
\end{figure}

\begin{figure}[tbph]
\centering \includegraphics[width=4in]{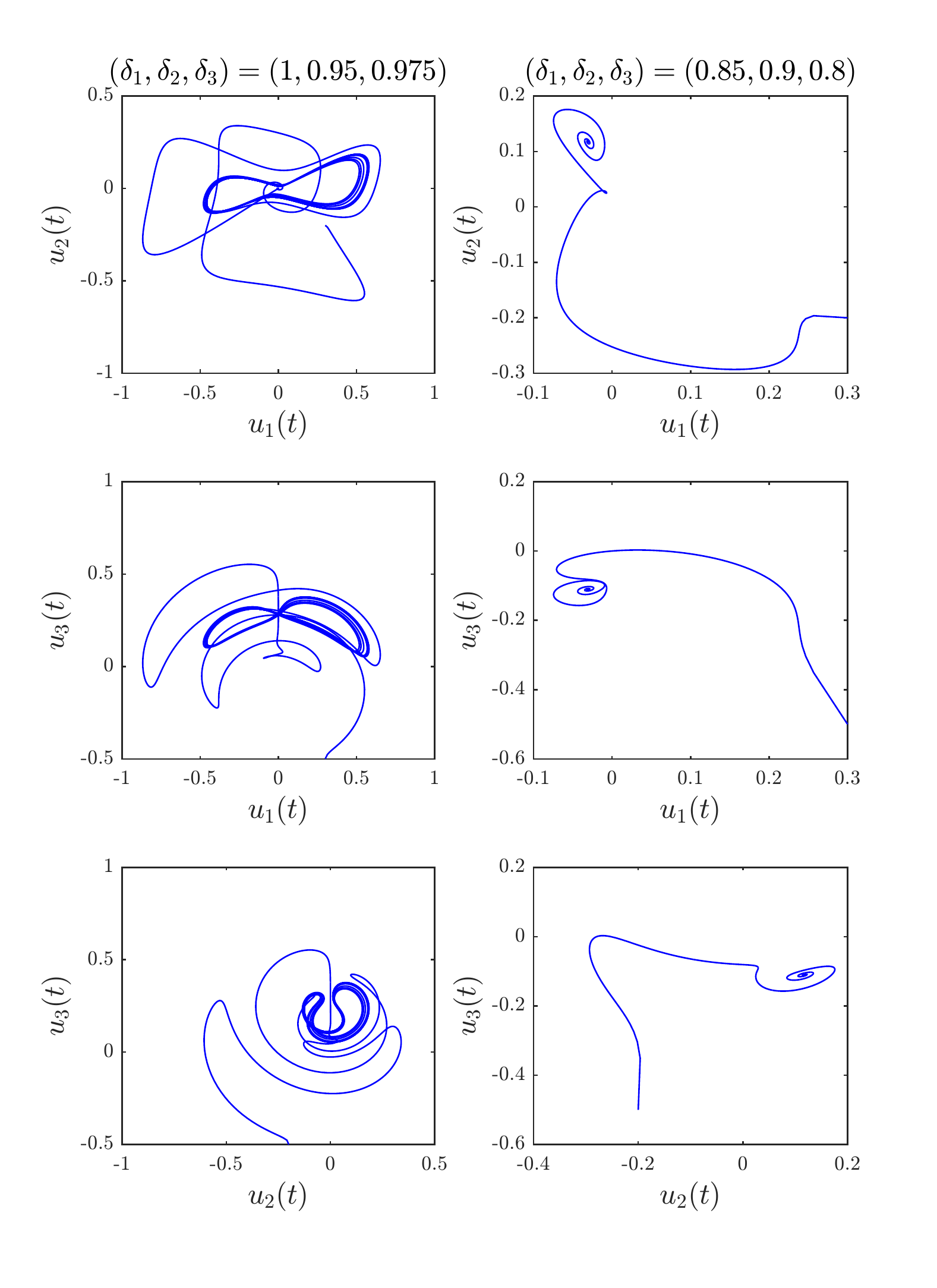}
\caption{Phase--space portraits of the fractional incommensurate
Newton--Leipnik chaotic system for $\left( a,\protect \alpha \right) =\left(
0.4,0.175\right) $ and initial conditions $\left( u_{1}\left( 0\right)
,u_{2}\left( 0\right) ,u_{3}\left( 0\right) \right) =\left(
0.349,0,-0.3\right) $ with different fractional orders.}
\label{Figure3}
\end{figure}

\section{Time--Fractional Reaction--Diffusion Model}

In this paper, we are concerned with the time--fractional
reaction--diffusion Newton--Leipnik system given by%
\begin{equation}
\left \{ 
\begin{array}{l}
_{t_{0}}^{C}D_{t}^{\delta _{1}}u_{1}-d_{1}\Delta
u_{1}=-au_{1}+u_{2}+10u_{2}u_{3},\  \  \  \  \  \text{in }\mathbb{R}^{+}\times
\Omega , \\ 
_{t_{0}}^{C}D_{t}^{\delta _{2}}u_{2}-d_{2}\Delta
u_{2}=-u_{1}-0.4u_{2}+5u_{1}u_{3},\  \  \  \text{in }\mathbb{R}^{+}\times
\Omega , \\ 
_{t_{0}}^{C}D_{t}^{\delta _{3}}u_{3}-d_{3}\Delta u_{3}=\alpha
u_{3}-5u_{1}u_{2},\  \  \  \  \  \  \  \  \  \  \  \  \  \  \  \text{in }\mathbb{R}%
^{+}\times \Omega ,%
\end{array}%
\right.  \label{3.1}
\end{equation}%
where $u_{i}\left( x,t\right) ,i=1,2,3,$\ are the spatio--temporal states of
the system, $\Omega $ is a bounded domain in $%
\mathbb{R}
^{n}$ with smooth boundary $\partial \Omega $, $\Delta =\underset{i=1}{%
\overset{n}{\sum }}\frac{\partial ^{2}}{\partial x_{i}^{2}}$ is the
Laplacian operator on $\Omega $, and $d_{i}>0,i=1,2,3,$\ are the diffusivity
constants for each of the states. We assume the nonnegative initial
conditions%
\begin{equation}
0\leq u_{i}\left( 0,x\right) =u_{i,0}\left( x\right) ,\text{ }i=1,2,3,\text{%
\  \  \  \ in }\Omega ,  \label{3.2}
\end{equation}%
with $u_{i,0}\in C^{2}\left( \Omega \right) \cap C\left( \overline{\Omega }%
\right) $, and impose homogoneous Neumann boundary conditions%
\begin{equation}
\dfrac{\partial u_{i}}{\partial \nu }=0,\text{ }i=1,2,3,\  \  \text{\  \  \  \  \
on \  \  \ }\mathbb{R}^{+}\times \partial \Omega ,  \label{3.3}
\end{equation}%
where $\nu $ is the unit outer normal to $\partial \Omega $.

Subject to certain parameter values and fractional orders, system (\ref{3.1}%
) can be shown to exhibit spatio--temporal chaos. For instance, choosing
parameters $\left( a,\alpha \right) =\left( 0.4,0.175\right) $ and initial
conditions%
\begin{equation}
\left \{ 
\begin{array}{l}
u_{1}\left( x,0\right) =0.349\left[ 1+0.3\cos \left( \frac{x}{2}\right) %
\right] , \\ 
u_{2}\left( x,0\right) =0, \\ 
u_{3}\left( x,0\right) =-0.3\left[ 1+0.3\cos \left( \frac{x}{2}\right) %
\right] ,%
\end{array}%
\right.   \label{3.3.0}
\end{equation}%
yields spatio--temporal chaos as shown in Figure \ref{Figure4}(left) for $%
\delta =0.99$. Figure \ref{Figure4}(left) was produced over the ranges $t\in %
\left[ 0,50\right] $ and $x\in \left[ 0,20\right] $. Although the chaotic
nature of the states is apparent, it always helps to visualize the
phase--space portraits of the system. In order to be able to do that, we
choose the single spatial point $x=10$ and plot its phase space over time.
The result is depicted in Figure \ref{Figure4}(right). It is interesting to
realize that the fractional order has an impact on the dynamics of the
system. Keeping the same parameters and initial conditions and changing the
fractional order to $\delta =0.99$\ yields the results depicted in Figure %
\ref{Figure6}. The trajectories clearly converge to a closed orbit, which
implies an oscillatory behavior. Reducing the order further to $\delta =0.90$%
\ yields an asymptotically stable solution as shown in Figure \ref{Figure8}.

\begin{figure}[tbph]
\centering \includegraphics[width=2.5in]{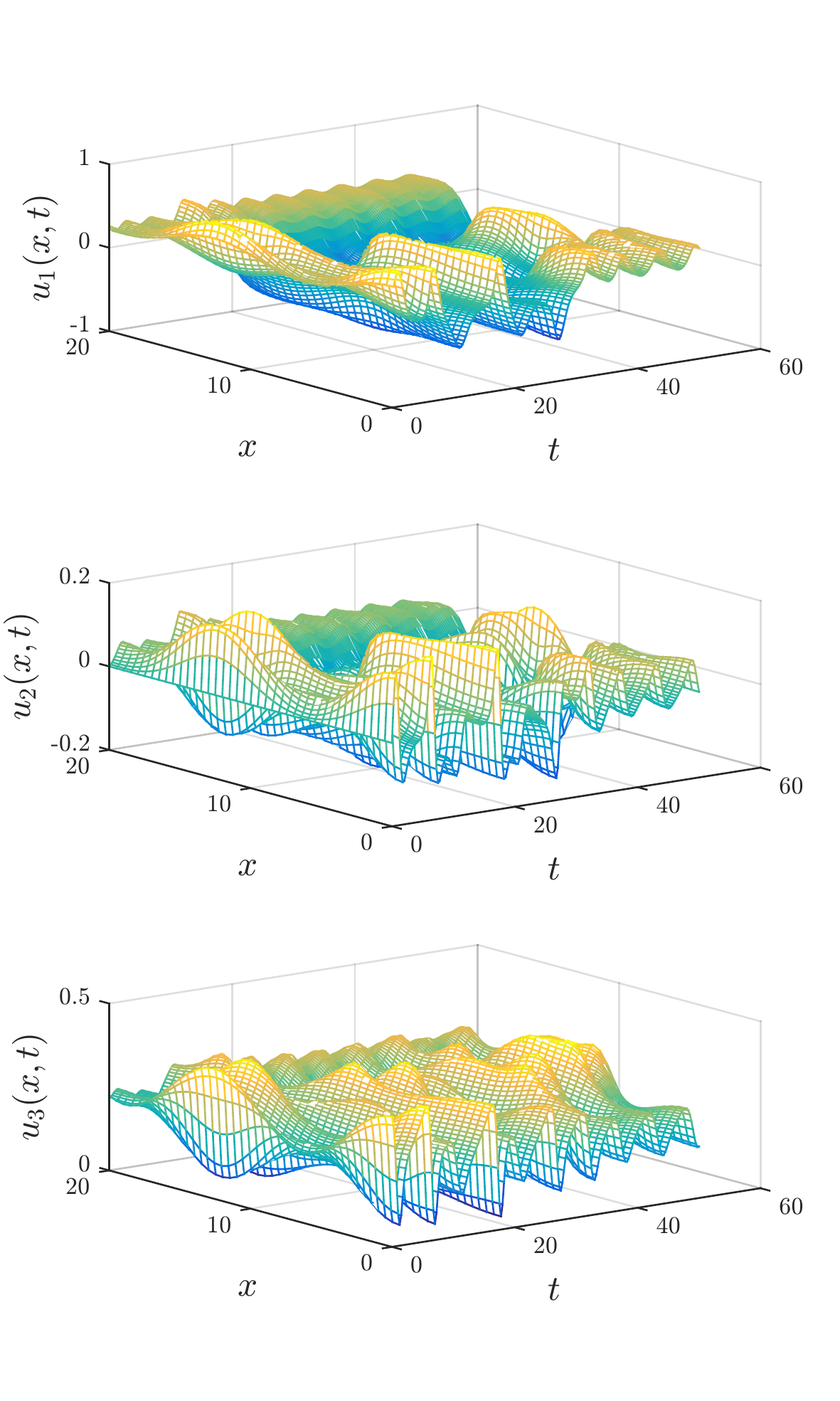} %
\includegraphics[width=2.5in]{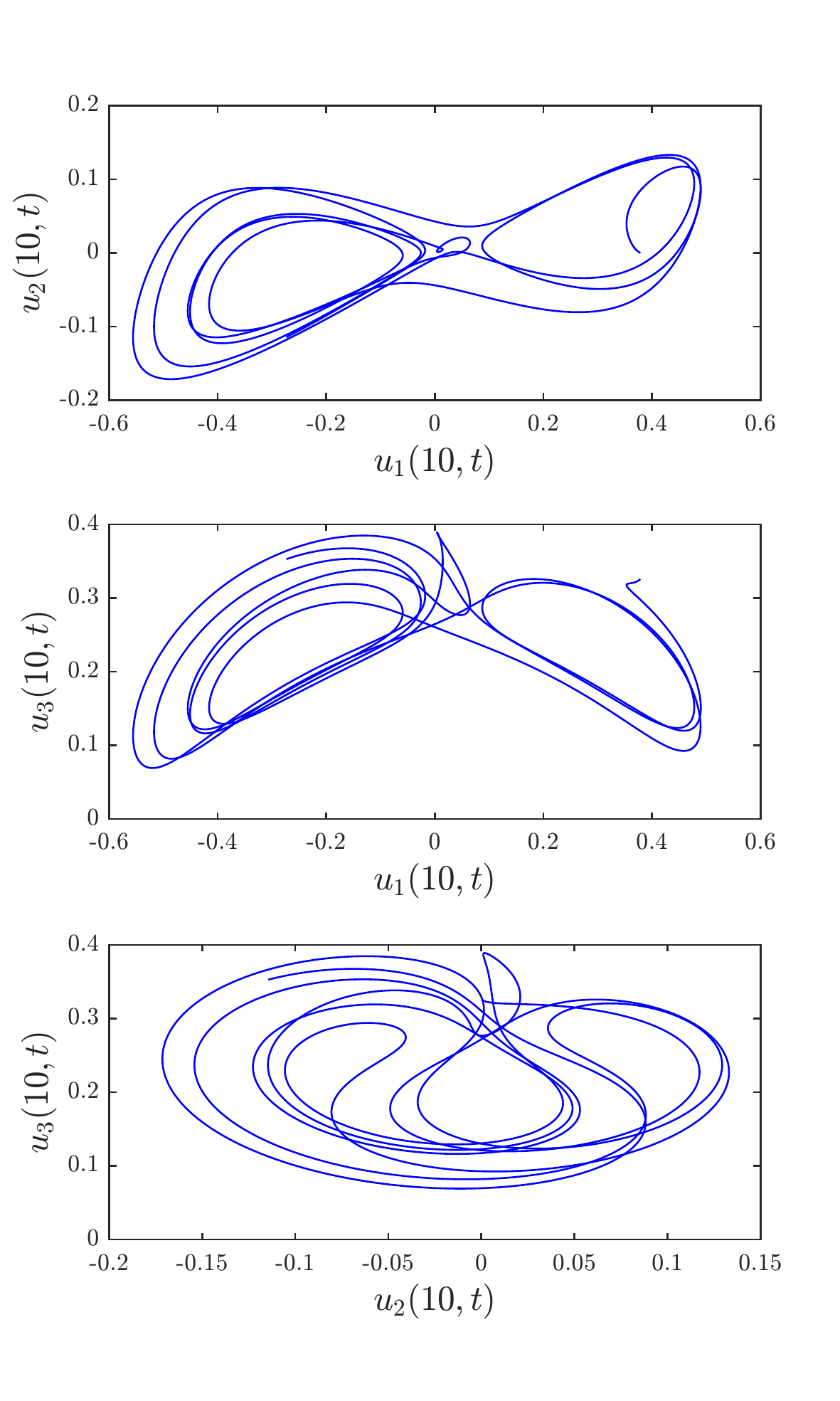}
\caption{Time evolution of the spatio--temporal states and the phase
portraits taken at $x=10$\ for parameters $\left( a,\protect \alpha \right)
=\left( 0.4,0.175\right) $, initial conditions (\protect \ref{3.3.0}), and
fractional order $\protect \delta =0.99$.}
\label{Figure4}
\end{figure}

\begin{figure}[tbph]
\centering%
\includegraphics[width=2.5in]{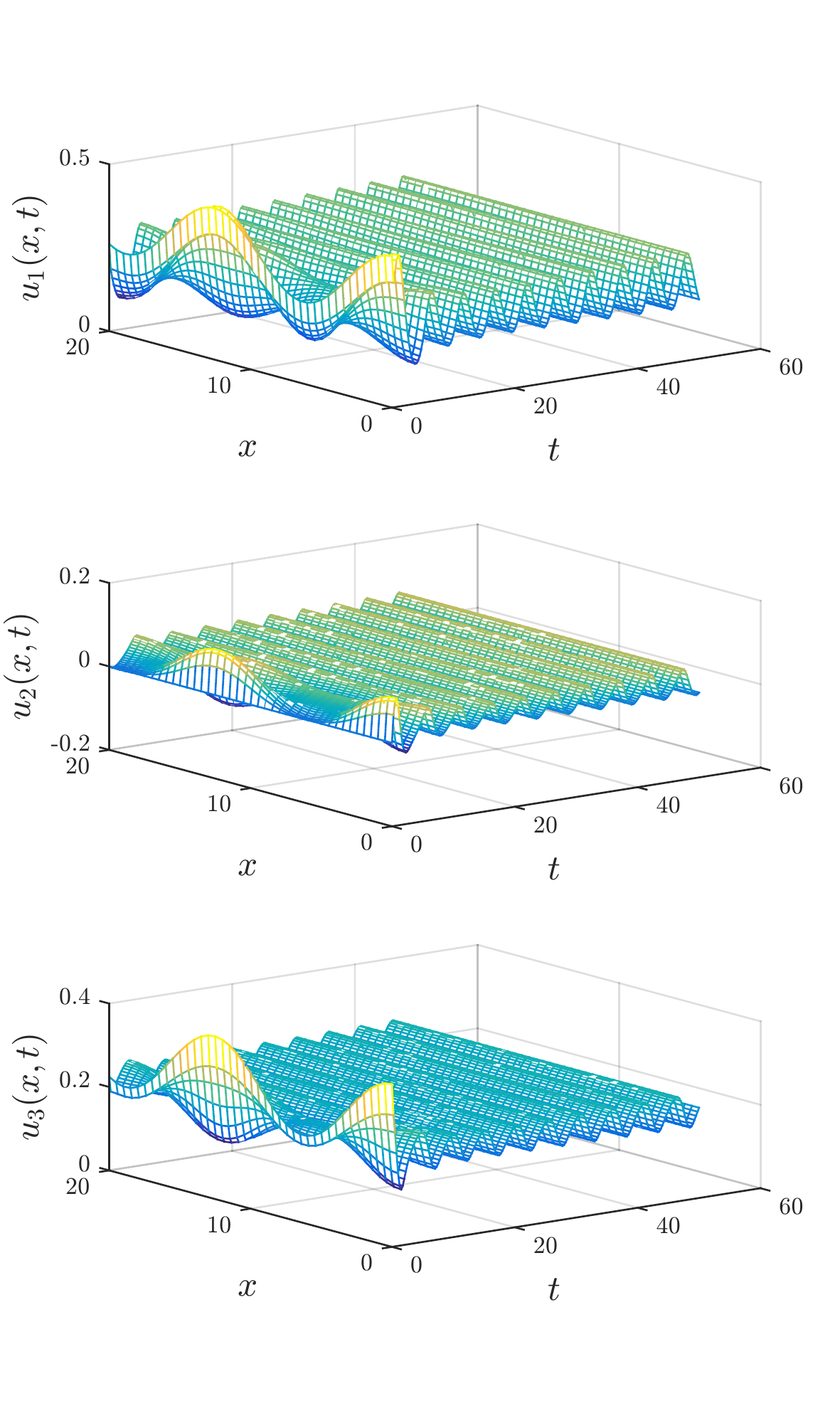}
\includegraphics[width=2.5in]{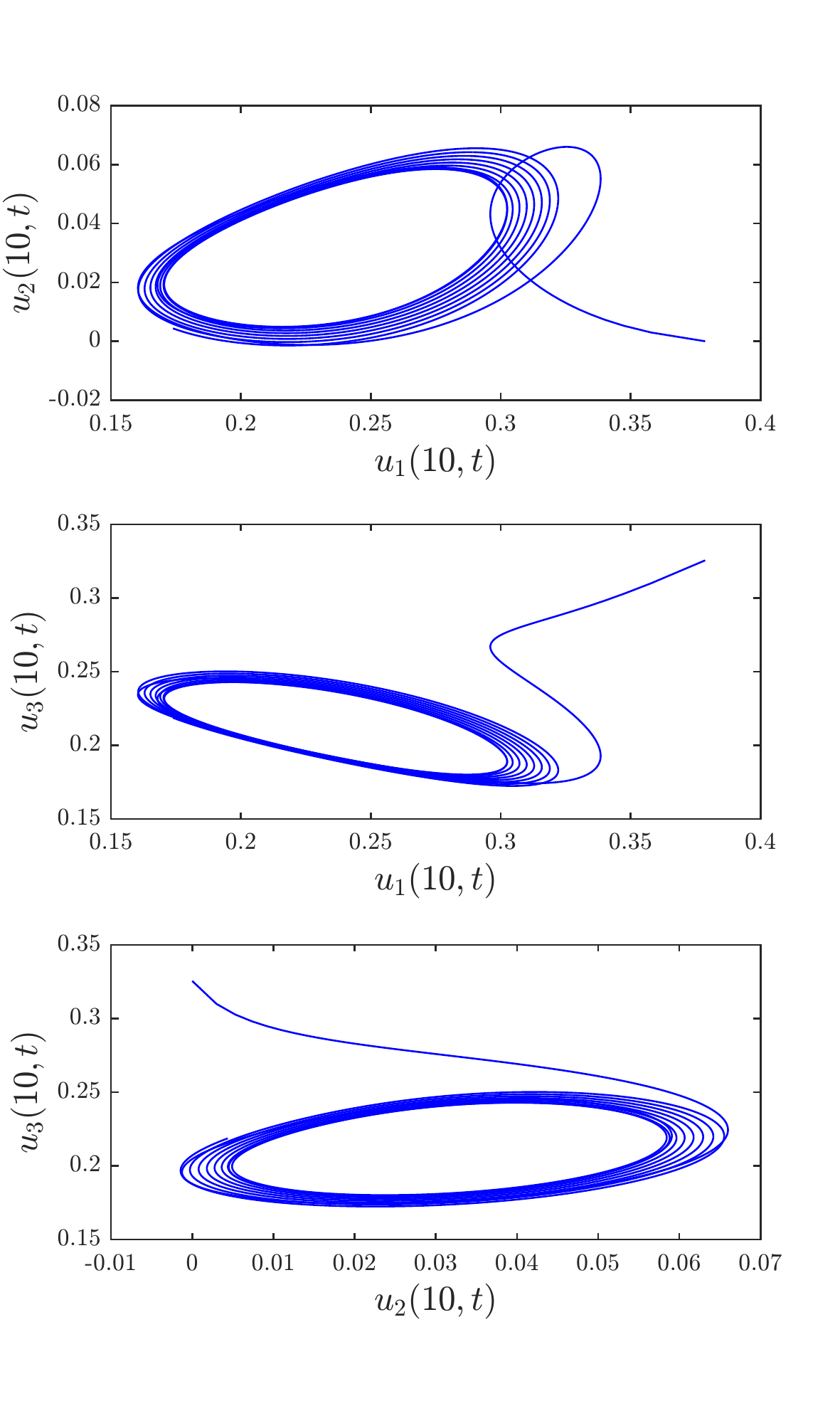}
\caption{Time evolution of the spatio--temporal states and the phase
portraits taken at $x=10$\ for parameters $\left( a,\protect \alpha \right)
=\left( 0.4,0.175\right) $, initial conditions (\protect \ref{3.3.0}), and
fractional order $\protect \delta =0.95$.}
\label{Figure6}
\end{figure}

\begin{figure}[tbph]
\centering%
\includegraphics[width=2.5in]{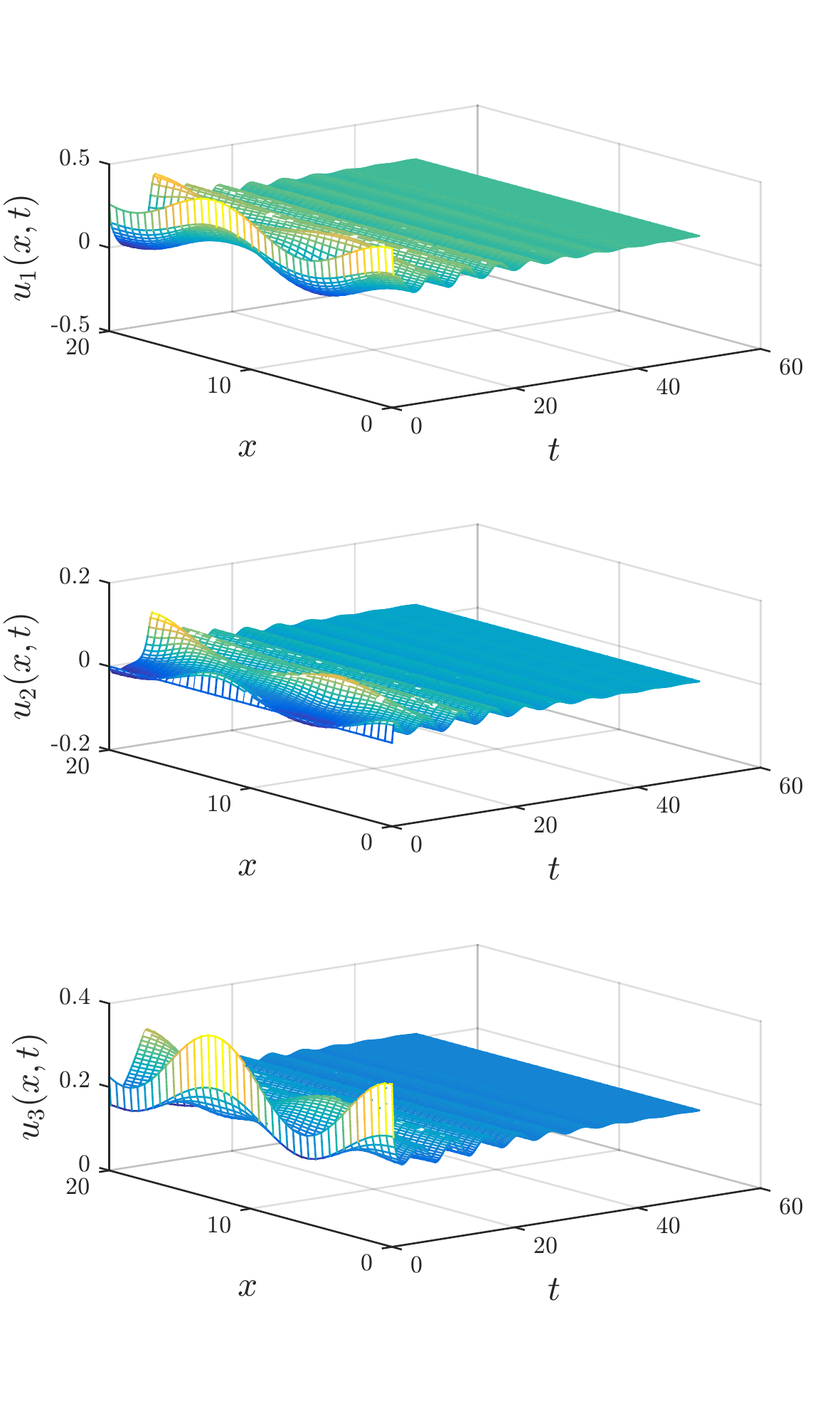}
\includegraphics[width=2.5in]{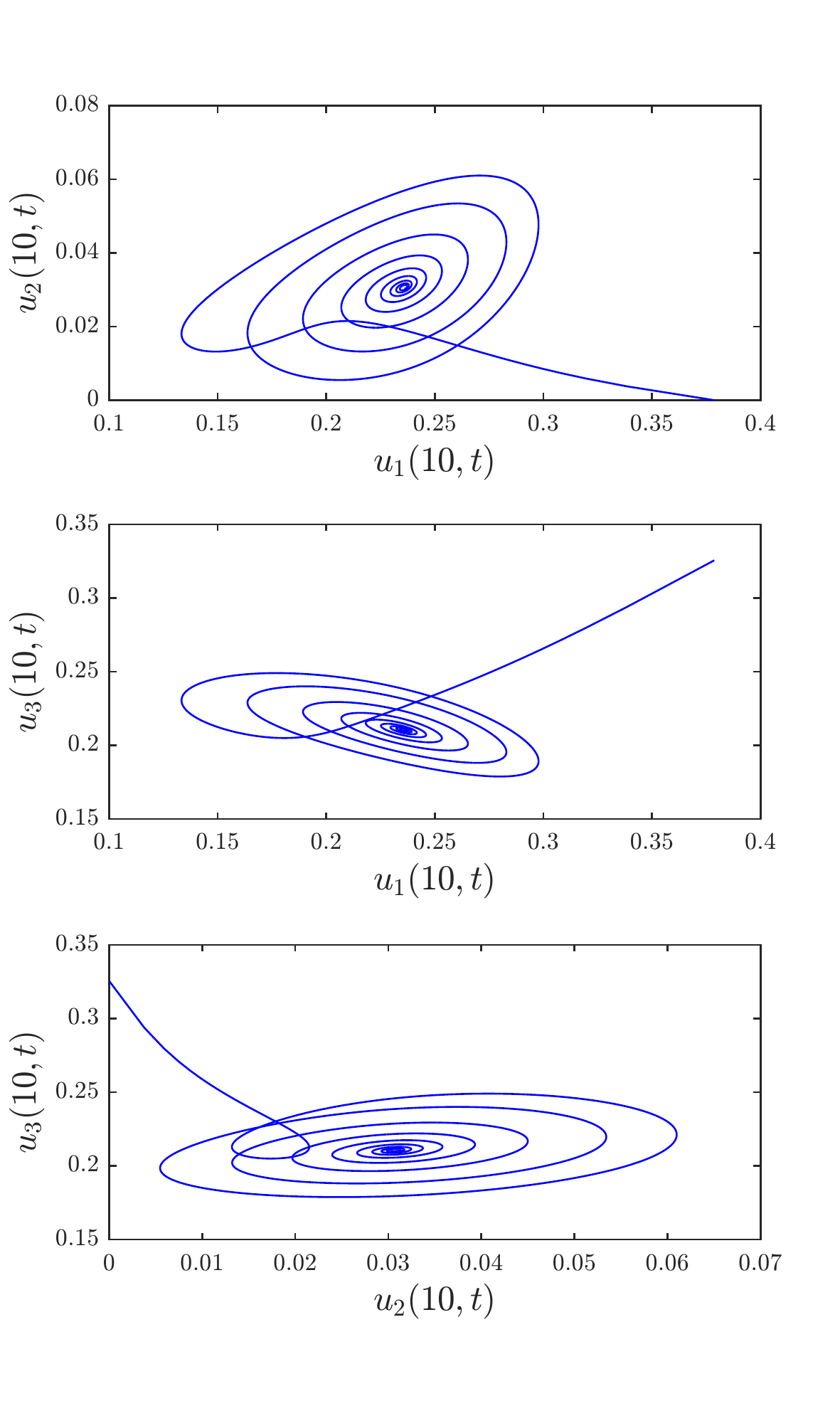}
\caption{Time evolution of the spatio--temporal states and the phase
portraits taken at $x=10$\ for parameters $\left( a,\protect \alpha \right)
=\left( 0.4,0.175\right) $, initial conditions (\protect \ref{3.3.0}), and
fractional order $\protect \delta =0.90$.}
\label{Figure8}
\end{figure}

For the incommensurate case, we consider the set of fractional constants $%
\left( \delta _{1},\delta _{2},\delta _{3}\right) =\left(
0.97,0.98,0.99\right) $. The spatio--temporal states and the phase space at $%
x=10$ are depicted in Figure \ref{Figure10}. The chaotic behavior with a
double strange attractor is apparent.

\begin{figure}[tbph]
\centering%
\includegraphics[width=2.5in]{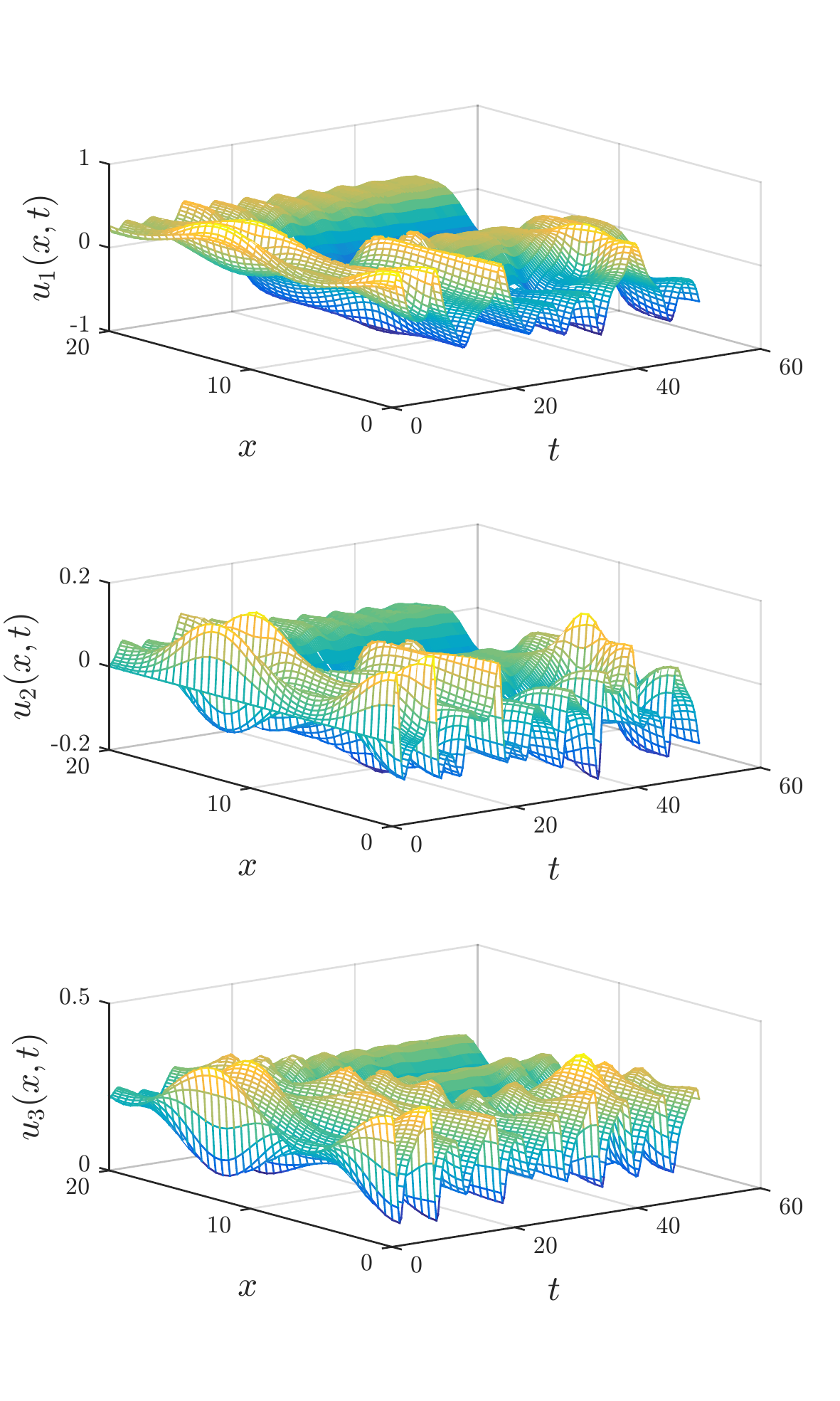}
\includegraphics[width=2.5in]{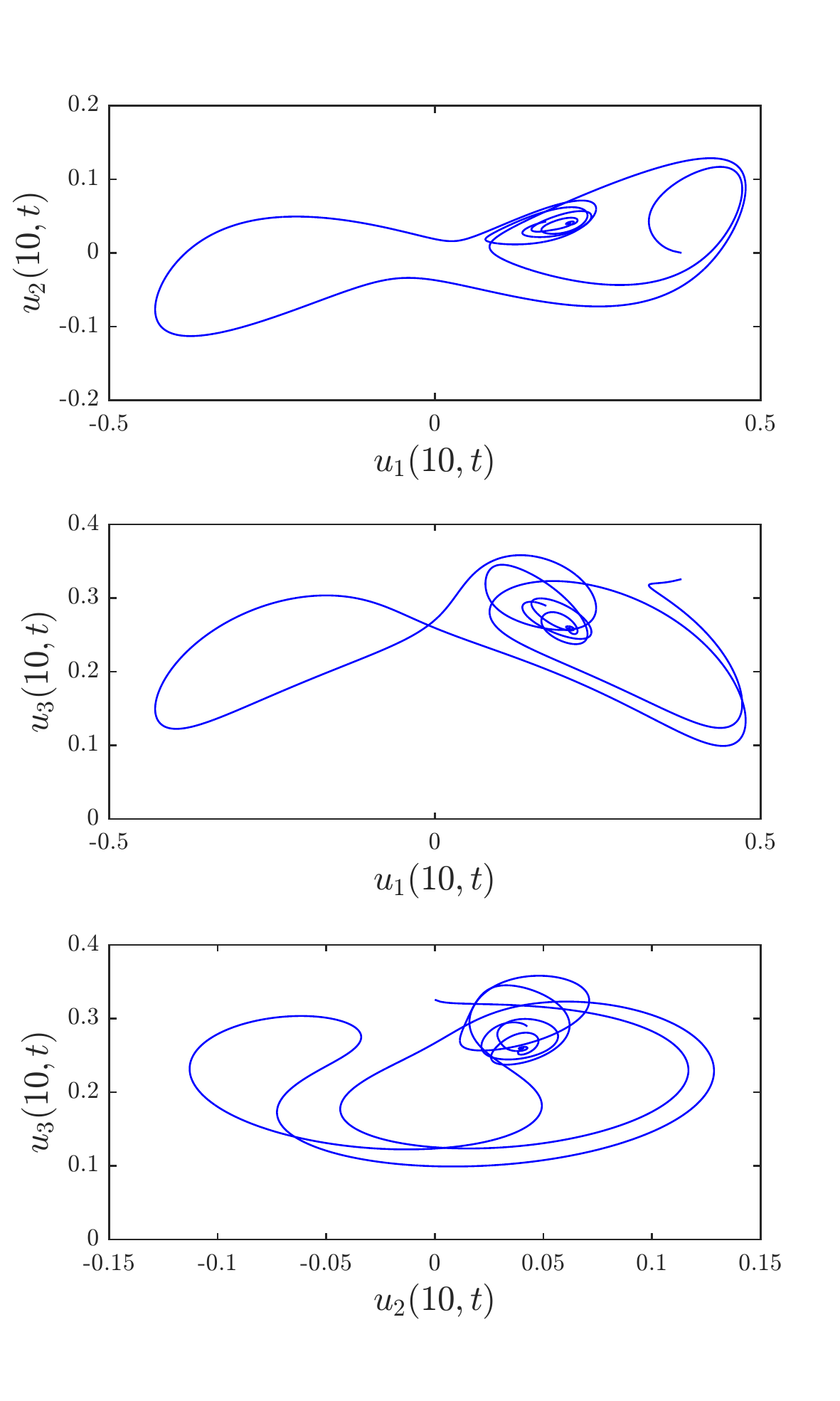}
\caption{Time evolution of the spatio--temporal states and the phase
portraits taken at $x=10$\ for parameters $\left( a,\protect \alpha \right)
=\left( 0.4,0.175\right) $, initial conditions (\protect \ref{3.3.0}), and
fractional order $\left( \protect \delta _{1},\protect \delta _{2},\protect%
\delta _{3}\right) =\left( 0.97,0.98,0.99\right) $.}
\label{Figure10}
\end{figure}

\section{Complete Synchronization\label{SecSync}}

The main objective of our paper is to develop an adaptive control scheme to
synchronize an identical slave system of the form%
\begin{equation}
\left \{ 
\begin{array}{l}
_{t_{0}}^{C}D_{t}^{\delta _{1}}v_{1}-d_{1}\Delta
v_{1}=-av_{1}+v_{2}+10v_{2}v_{3}+\phi _{1}, \\ 
_{t_{0}}^{C}D_{t}^{\delta _{2}}v_{2}-d_{2}\Delta
v_{2}=-v_{1}-0.4v_{2}+5v_{1}v_{3}+\phi _{2}, \\ 
_{t_{0}}^{C}D_{t}^{\delta _{3}}v_{3}-d_{3}\Delta v_{3}=\alpha
v_{3}-5v_{1}v_{2}+\phi _{3},%
\end{array}%
\right.  \label{3.4}
\end{equation}%
with $v_{i}\left( x,t\right) ,i=1,2,3,$ denoting the states of the slave
system and $\phi _{i}\left( x,t\right) ,i=1,2,3,$ being some control
parameters, to the master system given in (\ref{3.1}). Before we present the
synchronization scheme and assess the local and global asymptotic stability
of the zero solution to the error system, let us define the goal of
synchronization.

\begin{definition}
\label{Def5}System (\ref{3.1}) and the controlled system (\ref{3.4}) are
said to be asymptotically synchronized if 
\begin{equation}
\lim_{t\rightarrow \infty }\left \Vert u-v\right \Vert =0  \label{4.1}
\end{equation}%
for any $t>0$, where $u=(u_{1},u_{2},u_{3})^{T}\in \left( \mathbb{R}%
^{+}\times \Omega \right) ^{3}$ and $v=(v_{1},v_{2},v_{3})^{T}\in \left( 
\mathbb{R}^{+}\times \Omega \right) ^{3}$.
\end{definition}

We also need to define some necessary notation relating to the eigenvalues
and eigenfunctions of the Laplacian operator. We denote the eigenvalues of
the elliptic operator ($-\Delta $) subject to the homogeneous Neumann
boundary conditions on $\Omega $ by%
\begin{equation}
0=\lambda _{0}<\lambda _{1}\leq \lambda _{2}\leq ...  \label{4.2}
\end{equation}%
We assume that each eigenvalue $\lambda _{i}$ has multiplicity $m_{i}\geq 1$%
. We also denote the normalized eigenfunctions corresponding to $\lambda
_{i} $ by $\Phi _{ij},1\leq j\leq m_{i}$. It should be noted that $\Phi _{0}$
is a constant and $\lambda _{i}\rightarrow \infty $ as $i\rightarrow \infty $%
. The eigenfunctions and eigenvalues possess a number of interesting
properties including%
\begin{equation}
\begin{array}{lll}
-\Delta \Phi _{ij}=\lambda _{i}\Phi _{ij} & \text{in} & \Omega , \\ 
\frac{\partial \Phi _{ij}}{\partial \nu }=0 & \text{on} & \partial \Omega ,
\\ 
\int_{\Omega }\Phi _{ij}^{2}\left( x\right) dx=1. &  & 
\end{array}
\label{4.3}
\end{equation}

We are ready to present our main result as stated in the following theorem.
We assume that the fractional order is identical for all components of the
master and slave systems yielding a commensurate system. The local
asymptotic convergence of the synchronization error is established by means
of eigenfunction analysis and the global convergence is guaranteed by the
Lyapunov method. For the synchronization problem, we will assume $\delta
_{1}=\delta _{2}=\delta _{3}=:\delta $.

\begin{theorem}
\label{Theo2}For the general fractional orders $0<\delta \leq 1$, the
master--slave pair (\ref{3.1}--\ref{3.4}) is globally synchronized subject
to the nonlinear control laws%
\begin{equation}
\left \{ 
\begin{array}{l}
\phi _{1}=-10\left( e_{2}e_{3}+u_{2}e_{3}+e_{2}u_{3}\right) , \\ 
\phi _{2}=-5\left( e_{1}e_{2}+u_{1}e_{3}+e_{1}u_{3}\right) , \\ 
\phi _{3}=5\left( e_{1}e_{2}+u_{1}e_{2}+e_{1}u_{2}\right) -\left( \alpha
+0.4\right) e_{3},%
\end{array}%
\right.  \label{4.4}
\end{equation}%
if for all eigenvalues $\lambda _{i}$ satisfying%
\begin{equation}
\lambda _{i}<\frac{2}{\left \vert d_{1}-d_{2}\right \vert },\
d_{1}\not=d_{2},  \label{4.5}
\end{equation}%
condition%
\begin{equation}
\left \vert \arg \left( \xi _{1,2}\right) \right \vert >\frac{\delta \pi }{2}
\label{4.6}
\end{equation}%
is fulfilled, where%
\begin{equation}
\xi _{1,2}=\frac{1}{2}\left[ \left( -\left( d_{1}+d_{2}\right) \lambda
_{i}-0.8\right) \pm i\sqrt{4-\left( d_{1}-d_{2}\right) ^{2}\lambda _{i}^{2}}%
\right] .  \label{4.7}
\end{equation}
\end{theorem}

\begin{proof}
\textbf{Part\ I:} In the first part of our proof, we show that the zero
solution of the error system is locally asymptotically stable in the
diffusion free case. The synchronization errors can be given by%
\begin{equation}
\left \{ 
\begin{array}{l}
D_{t}^{\delta }e_{1}-d_{1}\Delta e_{1}=-0.4e_{1}+e_{2}+10\left(
v_{2}v_{3}-u_{2}u_{3}\right) +\phi _{1}, \\ 
D_{t}^{\delta }e_{2}-d_{2}\Delta e_{2}=-e_{1}-0.4e_{2}+5\left(
v_{1}v_{3}-u_{1}u_{3}\right) +\phi _{2}, \\ 
D_{t}^{\delta }e_{3}-d_{3}\Delta e_{3}=\alpha e_{3}-5\left(
v_{1}v_{2}-u_{1}u_{2}\right) +\phi _{3}.%
\end{array}%
\right.  \label{4.8}
\end{equation}%
This can be rewritten in the more compact form%
\begin{equation}
\left \{ 
\begin{array}{l}
D_{t}^{\delta }e_{1}-d_{1}\Delta e_{1}=-0.4e_{1}+e_{2}+10\left(
e_{2}e_{3}+u_{2}e_{3}+e_{2}u_{3}\right) +\phi _{1}, \\ 
D_{t}^{\delta }e_{2}-d_{2}\Delta e_{2}=-e_{1}-0.4e_{2}+5\left(
e_{1}e_{2}+u_{1}e_{3}+e_{1}u_{3}\right) +\phi _{2}, \\ 
D_{t}^{\delta }e_{3}-d_{3}\Delta e_{3}=\alpha e_{3}-5\left(
e_{1}e_{2}+u_{1}e_{2}+e_{1}u_{2}\right) +\phi _{3}.%
\end{array}%
\right.  \label{4.9}
\end{equation}%
Subsituting the controls (\ref{4.4}) in (\ref{4.9}) yields the error dynamics%
\begin{equation}
\left \{ 
\begin{array}{l}
D_{t}^{\delta }e_{1}-d_{1}\Delta e_{1}=-0.4e_{1}+e_{2},\  \  \  \  \  \text{in }%
\mathbb{R}^{+}\times \Omega , \\ 
D_{t}^{\delta }e_{2}-d_{2}\Delta e_{2}=-e_{1}-0.4e_{2},\  \  \  \  \  \text{in }%
\mathbb{R}^{+}\times \Omega , \\ 
D_{t}^{\delta }e_{3}-d_{3}\Delta e_{3}=-0.4e_{3},\  \  \  \  \  \  \  \  \  \  \  \  \ 
\text{in }\mathbb{R}^{+}\times \Omega ,%
\end{array}%
\right.  \label{4.10}
\end{equation}%
with the Jacobian matrix%
\begin{equation}
J_{e}=\left( 
\begin{array}{ccc}
-0.4 & 1 & 0 \\ 
-1 & -0.4 & 0 \\ 
0 & 0 & -0.4%
\end{array}%
\right) .  \label{4.11}
\end{equation}%
The eigenvalues of $J_{e}$\ are simply $-0.4+1.0i$, $-0.4-1.0i$, and $-0.4$.
We see that%
\begin{equation*}
\left \vert \arg \left( -0.4\pm i\right) \right \vert =1.9513,
\end{equation*}%
and%
\begin{equation*}
\left \vert \arg \left( -0.4\right) \right \vert =\pi .
\end{equation*}

Selecting%
\begin{equation*}
\delta <1.2422
\end{equation*}%
guarantees asymptotic stability. Since is assumed to lie in the interval $%
0<\delta \leq 1$, local asymptotic stability of the zero solution to (\ref%
{4.9}) in the diffusion free case is evident.

\textbf{Part II:} In this second part, we want to include diffusion and
assess the local stability of the zero solution. In the presence of
diffusion, the steady state solution satisfies the following system%
\begin{equation*}
\left \{ 
\begin{array}{l}
-d_{1}\Delta e_{1}=-0.4e_{1}+e_{2}, \\ 
-d_{2}\Delta e_{2}=-e_{1}-0.4e_{2}, \\ 
-d_{3}\Delta e_{3}=-0.4e_{3},%
\end{array}%
\right.
\end{equation*}%
subject to the homogeneous Neumann boundary conditions 
\begin{equation*}
\dfrac{\partial e_{1}}{\partial \nu }=\dfrac{\partial e_{2}}{\partial \nu }=%
\dfrac{\partial e_{3}}{\partial \nu }=0\text{ for all\ }x\in \partial \Omega
.
\end{equation*}

Consider the linearization operator%
\begin{equation*}
L=\left( 
\begin{array}{ccc}
-d_{1}\Delta -0.4 & 1 & 0 \\ 
-1 & -d_{2}\Delta -0.4 & 0 \\ 
0 & 0 & -d_{3}\Delta -0.4%
\end{array}%
\right) .
\end{equation*}%
Let $\left( \phi \left( x\right) ,\psi \left( x\right) ,\Upsilon \left(
x\right) \right) $ be an eigenfunction of $L$ corresponding to the
eigenvalue $\xi $, i.e. the pair satisfies%
\begin{equation*}
L\left( \phi \left( x\right) ,\psi \left( x\right) ,\Upsilon \left( x\right)
\right) ^{t}=\xi \left( \phi \left( x\right) ,\psi \left( x\right) ,\Upsilon
\left( x\right) \right) ^{t}.
\end{equation*}%
Alternatively, we can write%
\begin{equation*}
\left[ L-\xi I\right] \left( \phi \left( x\right) ,\psi \left( x\right)
\right) ^{t}=\left( 0,0,0\right) ^{t},
\end{equation*}%
leading to%
\begin{equation*}
\left( 
\begin{array}{ccc}
-d_{1}\Delta -0.4-\xi & 1 & 0 \\ 
-1 & -d_{2}\Delta -0.4-\xi & 0 \\ 
0 & 0 & -d_{3}\Delta -0.4-\xi%
\end{array}%
\right) \left( 
\begin{array}{c}
\phi \\ 
\psi \\ 
\Upsilon%
\end{array}%
\right) =\left( 
\begin{array}{c}
0 \\ 
0 \\ 
0%
\end{array}%
\right) .
\end{equation*}%
Using the factorizations%
\begin{equation*}
\phi =\sum_{0\leq i\leq \infty ,1\leq j\leq m_{i}}a_{ij}\Phi _{ij}\text{ },\
\psi =\sum_{0\leq i\leq \infty ,1\leq j\leq m_{i}}b_{ij}\Phi _{ij},\  \text{%
and }\Upsilon =\sum_{0\leq i\leq \infty ,1\leq j\leq m_{i}}c_{ij}\Phi _{ij},
\end{equation*}%
the matrix equation can be formulated as%
\begin{equation*}
\sum_{0\leq i\leq \infty ,1\leq j\leq m_{i}}\left( 
\begin{array}{ccc}
-d_{1}\lambda _{i}-0.4-\xi & 1 & 0 \\ 
-1 & -d_{2}\lambda _{i}-0.4-\xi & 0 \\ 
0 & 0 & -d_{3}\lambda _{i}-0.4-\xi%
\end{array}%
\right) \left( 
\begin{array}{c}
a_{ij} \\ 
b_{ij} \\ 
c_{ij}%
\end{array}%
\right) \Phi _{ij}=\left( 
\begin{array}{c}
0 \\ 
0 \\ 
0%
\end{array}%
\right) .
\end{equation*}%
Disregarding the term $-\xi $, the stability of the steady state solution
relies on the eigenvalues of%
\begin{equation*}
A_{i}=\left( 
\begin{array}{ccc}
-d_{1}\lambda _{i}-0.4 & 1 & 0 \\ 
-1 & -d_{2}\lambda _{i}-0.4 & 0 \\ 
0 & 0 & -d_{3}\lambda _{i}-0.4%
\end{array}%
\right) ,
\end{equation*}%
whose characteristic polynomial is%
\begin{equation*}
\left( \left( d_{1}\lambda _{i}+0.4+\xi \right) \left( d_{2}\lambda
_{i}+0.4+\xi \right) +1\right) \left( d_{3}\lambda _{i}+0.4+\xi \right) =0.
\end{equation*}%
Clearly, one of the eigenvalues is%
\begin{equation*}
\xi _{3}=-d_{3}\lambda _{i}-0.4.
\end{equation*}%
The remaining eigenvalues are the solutions of%
\begin{equation*}
\left \vert 
\begin{array}{cc}
-d_{1}\lambda _{i}-0.4-\xi & 1 \\ 
-1 & -d_{2}\lambda _{i}-0.4-\xi%
\end{array}%
\right \vert =0,
\end{equation*}%
or more compactly%
\begin{equation*}
\xi ^{2}-\left( -\left( d_{1}+d_{2}\right) \lambda _{i}-0.8\right) \xi
+\left( \lambda _{i}^{2}d_{1}d_{2}+0.4\left( d_{1}+d_{2}\right) \lambda
_{i}+1.16\right) =0.
\end{equation*}%
The discriminant of this quandratic polynomial is%
\begin{eqnarray*}
\Delta &=&\left( -\left( d_{1}+d_{2}\right) \lambda _{i}-0.8\right)
^{2}-4\left( \lambda _{i}^{2}d_{1}d_{2}+0.4\left( d_{1}+d_{2}\right) \lambda
_{i}+1.16\right) \\
&=&\left( \left( d_{1}+d_{2}\right) ^{2}-4d_{1}d_{2}\right) \lambda
_{i}^{2}-4.0 \\
&=&\left( d_{1}-d_{2}\right) ^{2}\lambda _{i}^{2}-4.
\end{eqnarray*}%
Depending on the sign of $\Delta $, we may end up with different scenarios:

\begin{itemize}
\item First, if $\Delta =\left( d_{1}-d_{2}\right) ^{2}\lambda
_{i}^{2}-4\geq 0$, then the remaining two eigenvalues are both real. It
helps to consider the trace%
\begin{equation*}
\text{tr}\left( 
\begin{array}{cc}
-d_{1}\lambda _{i}-0.4 & 1 \\ 
-1 & -d_{2}\lambda _{i}-0.4%
\end{array}%
\right) ,
\end{equation*}%
which is clearly strictly negative for all $i\geq 0$, and the determinant%
\begin{equation*}
\det \left( 
\begin{array}{cc}
-d_{1}\lambda _{i}-0.4 & 1 \\ 
-1 & -d_{2}\lambda _{i}-0.4%
\end{array}%
\right) ,
\end{equation*}%
which is clearly strictly positive for all $i\geq 0$. Hence, $\xi _{1,2}\in 
\mathbb{R}
^{-}$.

\item If $\Delta =\left( d_{1}-d_{2}\right) ^{2}\lambda _{i}^{2}-4<0$, then%
\begin{equation*}
\lambda _{i}<\frac{2}{\left \vert d_{1}-d_{2}\right \vert },\
d_{1}\not=d_{2},.
\end{equation*}%
Hence, the two eigenvalues $\xi _{1,2}$ are complex and may be given by (\ref%
{4.7}).
\end{itemize}

This tells us that if all eigenvalues satisfying (\ref{4.5}) fulfill (\ref%
{4.6}), then the steady state solution is locally asymptotically stable.

\textbf{Part III:} Now that we have established sufficient conditions for
the local asymptotic stability of the zero solution to (\ref{4.9}), we move
to show that it is globally asymptotically stable. Consider the Lyapunov
function%
\begin{equation*}
V=\frac{1}{2}\int \left( e_{1}^{2}+e_{2}^{2}+e_{3}^{2}\right) dx.
\end{equation*}%
By taking the $\delta $ fractional derivative and employing Lemma \ref%
{Lemma4}, we obtain%
\begin{eqnarray*}
D_{t}^{\delta }V &=&\frac{1}{2}\int_{\Omega }\left( D_{t}^{\delta
}e_{1}^{2}+D_{t}^{\delta }e_{2}^{2}+D_{t}^{\delta }e_{3}^{2}\right) dx \\
&\leq &\int_{\Omega }e_{1}D_{t}^{\delta }e_{1}+e_{2}D_{t}^{\delta
}e_{2}+e_{3}D_{t}^{\delta }e_{3}dx \\
&\leq &I+J,
\end{eqnarray*}%
where%
\begin{equation*}
I=-d_{1}\int_{\Omega }\left \vert \nabla e_{1}\right \vert
^{2}dx-d_{2}\int_{\Omega }\left \vert \nabla e_{2}\right \vert
^{2}dx-d_{3}\int_{\Omega }\left \vert \nabla e_{3}\right \vert ^{2}dx<0,
\end{equation*}%
and%
\begin{equation*}
J=-\int_{\Omega }0.4e_{1}^{2}+0.4e_{2}^{2}+\left[ 0.4\right] e_{3}^{2}dx<0,
\end{equation*}

Hence, $D_{t}^{\delta }V<0$ and the zero solution of (\ref{4.9}) is globally
asymptotically stable. The proof is complete.
\end{proof}

\section{Numerical Results\label{SecRes}}

In order to verify the results of the previous section, we use numerical
simulations. We let $\left( a,\alpha \right) =\left( 0.4,0.175\right) $ and%
\begin{equation*}
\left \{ 
\begin{array}{l}
u_{1}\left( x,0\right) =0.349\left[ 1+0.3\cos \left( \frac{x}{2}\right) %
\right] , \\ 
u_{2}\left( x,0\right) =0, \\ 
u_{3}\left( x,0\right) =-0.3\left[ 1+0.3\cos \left( \frac{x}{2}\right) %
\right] .%
\end{array}%
\right. 
\end{equation*}%
Assuming $\delta =0.99$, Figure \ref{Figure12}(left) shows the
synchronization error between the master (\ref{3.1}) and slave (\ref{3.4})
for $\Omega \in \left[ 0,20\right] \times \left[ 0,50\right] $.
Synchronization is achieved by means of the 3D control law (\ref{4.4}). The
errors clearly decay to zero as time progresses indicating successful
synchronization. Figure \ref{Figure12}(right) shows the master and slave
trajectories in phase--space at spatial point $x=10$. The same experiment is
repeated with the different fractional orders%
\begin{equation*}
\left( \delta _{1},\delta _{2},\delta _{3}\right) =\left(
0.97,0.98,0.99\right) .
\end{equation*}%
The results are shown in Figure \ref{Figure14}. Again, as shown
analytically, the numerical results confirm the successful synchronization
of our master--slave pair.

\begin{figure}[tbph]
\centering%
\includegraphics[width=2.5in]{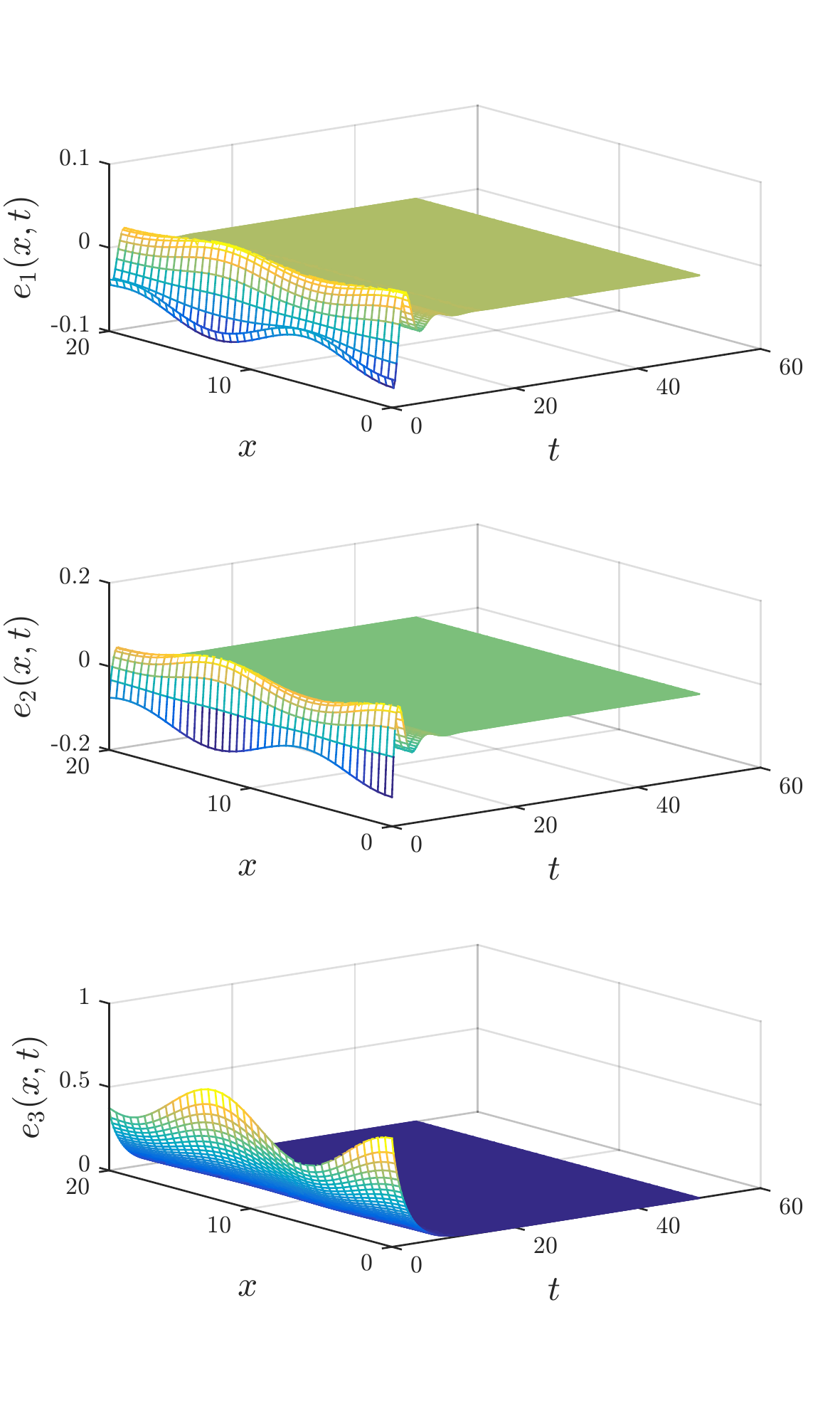}
\includegraphics[width=2.5in]{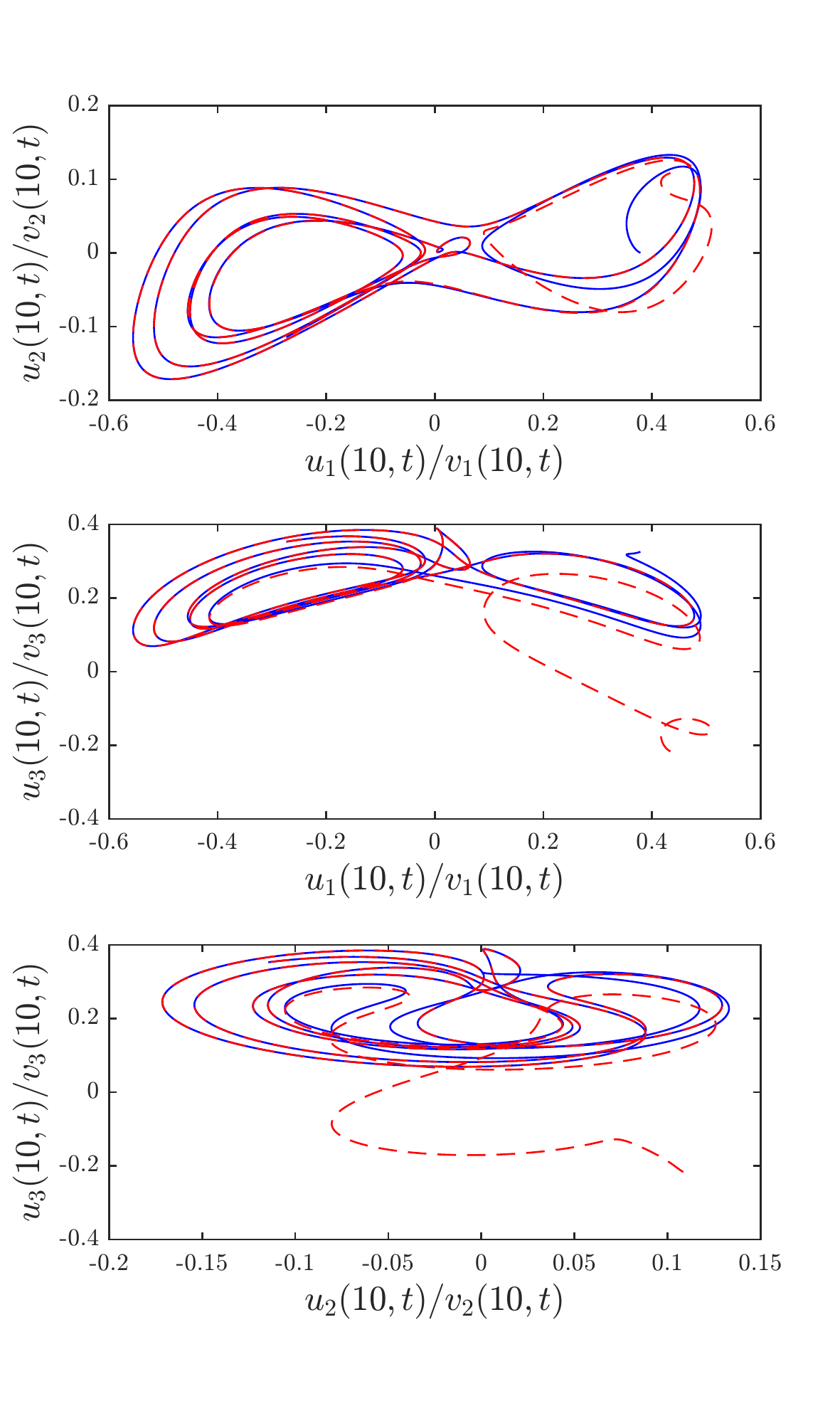}
\caption{Time evolution of the spatio--temporal synchronization errors
(left) and the phase portraits of the master (blue) and slave (red) taken at 
$x=10$ (right) with parameters $\left( a,\protect \alpha \right) =\left(
0.4,0.175\right) $, initial conditions (\protect \ref{3.3.0}), and fractional
order $\protect \delta =0.99$.}
\label{Figure12}
\end{figure}

\begin{figure}[tbph]
\centering%
\includegraphics[width=2.5in]{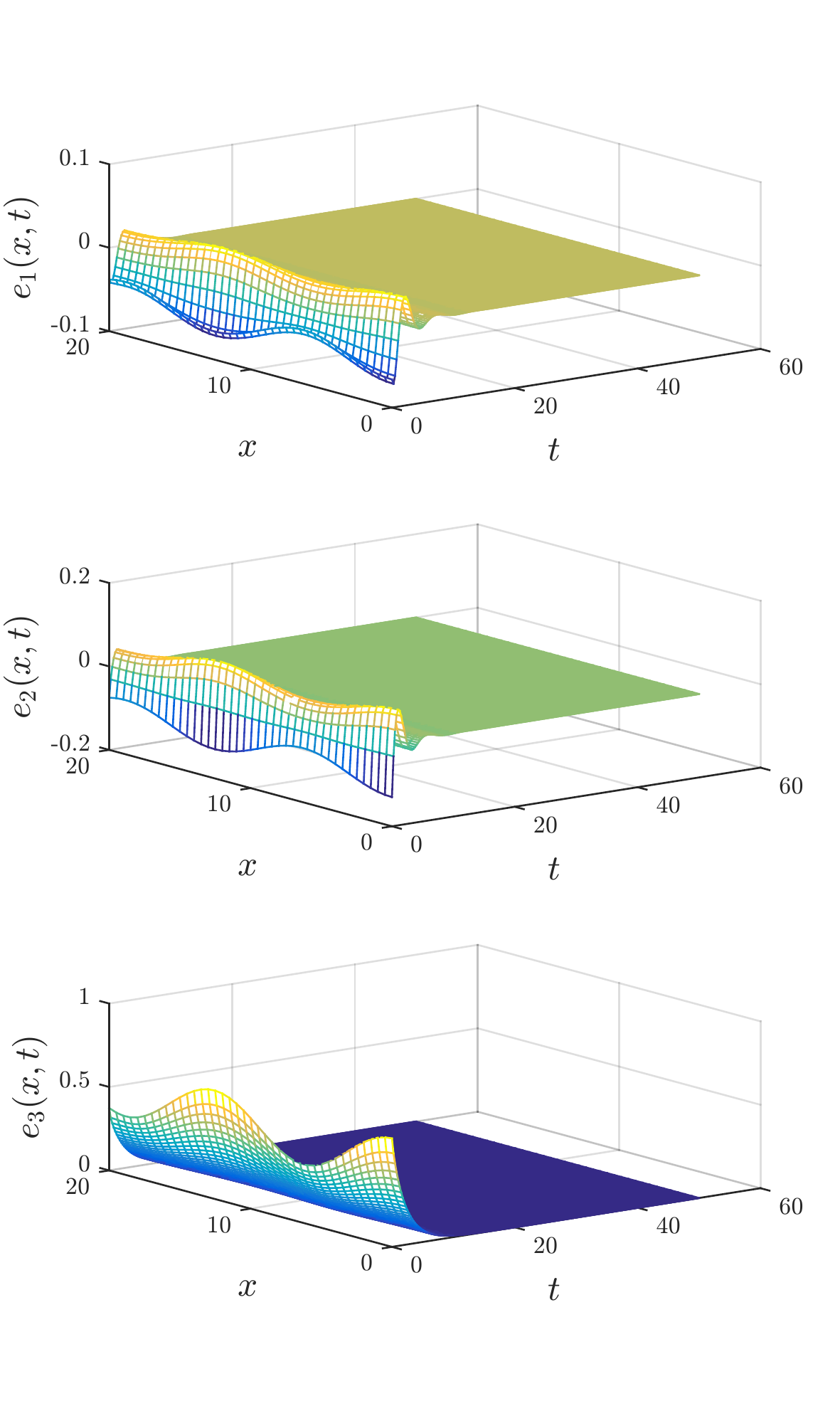}
\includegraphics[width=2.5in]{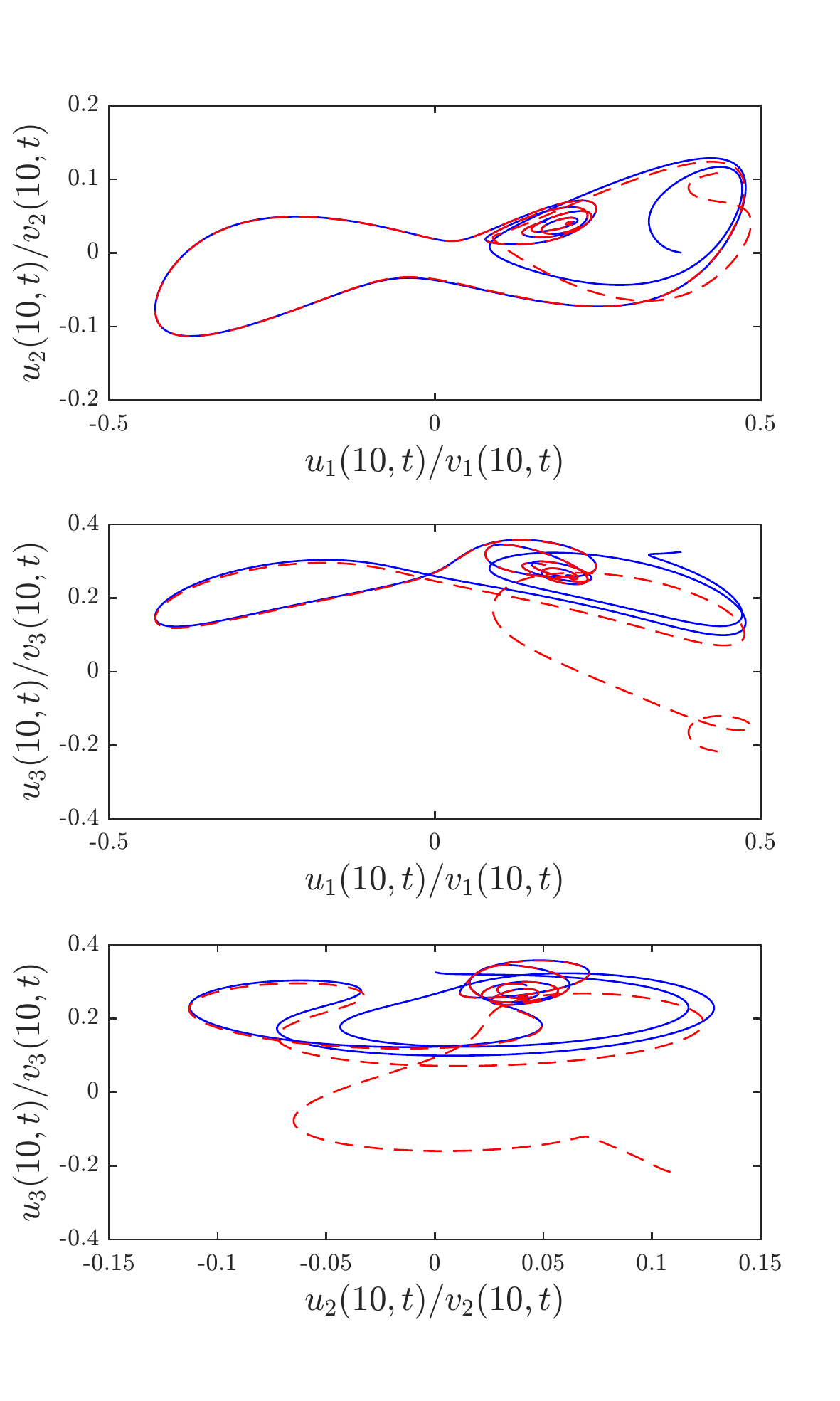}
\caption{Time evolution of the spatio--temporal synchronization errors
(left) and the phase portraits of the master (blue) and slave (red) taken at 
$x=10$ (right) with parameters $\left( a,\protect \alpha \right) =\left(
0.4,0.175\right) $, initial conditions (\protect \ref{3.3.0}), and fractional
orders $\left( \protect \delta _{1},\protect \delta _{2},\protect \delta %
_{3}\right) =\left( 0.97,0.98,0.99\right) $.}
\label{Figure14}
\end{figure}

\section{Concluding Remarks}

In this paper, we have considered a time--fractional spatio--temporal system
based on the Newton--Leipnik chaotic system. We started by giving a brief
overview of the most important definitions and theory related to fractional
dynamical system. Then, we reviewed some important aspects of the standard
and fractional Newton--Leipnik systems in the free diffusion scenario. The
main result of the paper concerns the global complete synchronization of a
master--slave pair of the proposed system. We established sufficient
conditions for the asymptotic convergence of the synchronization errors to
zero by means of local and global asymptotic stability methods. Throughout
the paper, we have used numerical simulations to illustrate the findings of
our study.

\end{document}